\documentclass[10pt,letter]{article}

\usepackage{amsmath}
\usepackage{amsfonts, amssymb, amsthm,times}
\usepackage{epsfig}
\usepackage{color}
\usepackage{multicol}

\newlength{\hchng}
\newlength{\vchng}
\newtheorem{thm}{Theorem}[section]
\newtheorem{prop}[thm]{Proposition}

\newtheorem{lemma}[thm]{Lemma}
\newtheorem{defn}[thm]{Definition}

\numberwithin{equation}{section}

\theoremstyle{definition}
\newtheorem{remark}[thm]{Remark}

\newcommand{\R}{\mathbb R}
\newcommand{\eps}{\varepsilon}

\newcommand{\grad} {\nabla}
\newcommand{\lap} {\triangle}

\newcommand{\dd} {\; \mathrm{d}}

\DeclareMathOperator{\dv}{div}

\newcommand{\MM}{\mathcal M}
\renewcommand{\AA}{\mathcal A}
\newcommand{\PP}{\mathcal P}

\newcommand{\DD}{\mathcal D}
\newcommand{\EE}{\mathcal E}
\newcommand{\II}{\mathcal I}
\newcommand{\RRd}{{\mathbb R}^n}

\title{H\"older continuity for a drift-diffusion equation with pressure}

\author{Luis Silvestre\footnote{Department of Mathematics, University of Chicago, 5734 University Avenue, Chicago, IL 60637. \texttt{luis@math.uchicago.edu}} \and Vlad Vicol\footnote{Department of Mathematics, University of Chicago, 5734 University Avenue, Chicago, IL 60637. \texttt{vicol@math.uchicago.edu}}}

\begin{document}
\maketitle

\begin{abstract}
We address the persistence of H\"older continuity for weak solutions of the linear drift-diffusion equation with nonlocal pressure
\begin{equation*}
u_t + b \cdot \grad u - \lap u = \grad p,\qquad \grad\cdot u =0
\end{equation*}on $[0,\infty) \times \R^{n}$, with $n \geq 2$. The drift velocity $b$ is assumed to be at the critical regularity level, with respect to the natural scaling of the equations. The proof draws on Campanato's characterization of H\"older spaces, and  uses a maximum-principle-type argument by which we control the growth in time of certain local averages of $u$. We provide an estimate that does not depend on any local smallness condition on the vector field $b$, but only on scale invariant quantities.
\end{abstract}

\section{Introduction}
\label{s:intro}

A classical problem in partial differential equations is to address the regularity of solutions to parabolic problems involving advection by a vector field $b$ and diffusion
\begin{equation} \label{e:no-pressure}
 u_t + b \cdot \grad u - \lap u = 0.
\end{equation}
If the vector field $b$ is sufficiently regular, the solution $u$ is expected to be regular as well. Naturally, this is expressed as a result of the type: if $b$ is bounded with respect to some norm, then $u$ is smooth in some sense. The appropriate norms for such statement are the ones that are either critical or subcritical with respect to the inherent scaling of the equation. More precisely, if $u$ is a solution of \eqref{e:no-pressure}, then for any $r>0$ the function $u_r(x,t) = r u(rx,r^{2}t)$ solves an equation of the same form but with drift velocity given by $b_r(x,t) = rb(rx,r^{2}t)$
. This change of variables acts as a \emph{zoom in} that focuses on the local behavior of the solution $u$. An assumption on $b$ is critical with respect to the scaling of the equation if the norm of $b_r$ coincides with the norm of $b$, for any $r>0$. The assumption would be subcritical if $b_r$ has smaller norm than $b$ for all small enough values of $r$, and supercritical otherwise. 

As a rule of thumb, with current methods it seems impossible to obtain a regularity result for \eqref{e:no-pressure} with a supercritical assumption on $b$, since the transport part of the equation would be stronger than diffusion in the small scales. With subcritical assumptions on $b$, it is generally possible to treat equation \eqref{e:no-pressure} as a perturbation of the heat equation, and strong regularity results in this direction are available. The Kato class condition for $b$ is probably the largest class that falls into this category. For results in the subcritical case see for example \cite{Aronson68,ChenZhao95,Zhang97}. 

Obtaining regularity estimates for \eqref{e:no-pressure} depending only on \emph{scale invariant} norms of $b$ requires the use of non-perturbative techniques, since the drift term does not become negligible at any scale. To the best of our knowledge, the only results available are variations of the De Giorgi-Nash-Moser Harnack inequality \cite{Moser61}, which states that weak solutions to \eqref{e:no-pressure} are $C^\alpha$ for positive time, for some small $\alpha>0$. Results in this direction include a variety of critical assumptions on the vector field $b$. For $b \in L^p_{t} L^q_{x}$ with $n/q+2/p=1$, we refer to \cite[Chapter 3]{LadyzhenskayaSollonikovUralceva67} or \cite{NazarovUraltseva10}. For divergence-free drift $b \in L^\infty_{t} BMO^{-1}_{x}$, the H\"older regularity of weak solutions was proved only recently in \cite{FriedlanderVicol10} and \cite{SereginSilvestreSverakZlatos10}. For $b$ in a space-time Morrey space, this result was obtained recently in \cite{NazarovUraltseva10}. See \cite{LiskevichZhang04,Semenov06,Zhang06} for other conditions on $b$ yielding H\"older regularity, such as the form boundedness condition.

The equation \eqref{e:no-pressure} is essentially a scalar equation, since even if $u$ is a vector field, each component would satisfy the same equation. In contrast, the De Giorgi-Nash-Moser theory is hard to apply to actual systems. In this article we consider a Stokes system with drift, i.e. we add a pressure term as is common for the equations of fluid dynamics, and we look for a solution $u$ which is divergence free. Given a divergence free vector field $b:[0,T]\times \R^n \to \R^n$, we consider the following evolution equation
\begin{equation} \label{e:linear-equation}
u_t + b \cdot \grad u - \lap u = \grad p
\end{equation}
for a solution $u : [0,T] \times \R^n \to \R^n$ which satisfies 
\begin{align}\label{e:divfree}
\nabla \cdot u = 0.
\end{align}
The term pressure gradient may be computed from \eqref{e:linear-equation}--\eqref{e:divfree} by the formula
\begin{align} \label{e:grad-p}
\grad p = \grad (-\lap)^{-1} \dv (b \cdot \grad u).
\end{align}

In this paper we prove that if a scale invariant norm of $b$ is bounded, then the $C^\alpha$ norm of $u$ at time $t$ is bounded in terms of its $C^\alpha$ norm at time zero. Our result is a propagation of regularity instead of a regularization result, in the sense that we require the initial data $u_0$ to be H\"older continuous, and $\alpha \in (0,1)$ is arbitrary.

The assumption on the divergence-free drift velocity $b$ is that it is an $L^{p}$ integrable function in time, with values in the $L^{1}$-based Morrey-Campanato space $M^{\beta}$, where $\beta \in [-1,1]$, and $p = 2/(1+\beta)$. We recall cf.~\cite[Definition 1.7.2]{Triebel92} the definition of the $L^{1}$-based Morrey-Campanato spaces $M^{s}$. For any $s \geq -n$, let $N = \max\{-1,[s]\}$. We say $f \in M^{s}$ if $f \in L^{1}_{loc}$ and 
\begin{align}
\sup_{x\in \R^{n}} \sup_{0<r<1} r^{-s} \inf_{P \in {\mathcal P}^{N}} \frac{1}{|B_{r}(x)|} \int_{B_{r}(x)} |f(z)-P(z)| \dd z < +\infty, \label{e:Morrey:def}
\end{align}
where ${\mathcal P}^{N}$ is the set of all polynomials of degree less than or equal to $N$, with the convention that ${\mathcal P}^{-1} = \{0\}$. In this article, we will only consider $\beta$ in the range $[-1,1]$.
The condition $b(\cdot,t) \in M^{\beta}$ has a different character depending of the value of the exponent $\beta$, cf.~\cite[Theorem 5.3.1]{Triebel92}: if $\beta=1$, $M^{\beta}$ coincides with the space of Lipschitz functions (here we subtract elements of ${\mathcal P}^{0}$ instead of ${\mathcal P}^{1}$, the later yielding the Zygmund class) ; if $\beta \in (0,1)$, it is exactly the H\"older class $C^\beta$; if $\beta=0$, it corresponds to the class of functions of bounded mean oscillation $BMO$; while if $\beta \in [-1,0)$ it is the usual Morrey-Campanato space. In all these cases, the estimate in our main theorem depends \emph{only} on the semi-norm $[b(\cdot,t)]_{M^{\beta}}$ associated to the space. In this paper we consider divergence-free drifts $b$ such that
\begin{align}
[b(\cdot,t)]_{M^{\beta}} =\sup_{x \in \R^{3}}\sup_{r>0} r^{-\beta}\MM(x,t,r) \leq g(t)\label{e:main}
\end{align}for some $g \in L^{p}([0,T])$,  where we define
\begin{align} 
\MM(x,t,r) =  \frac{1}{r^n} \int_{B_r(x)} |b(z,t)-\bar b(x,r,t)| \dd z = \int_{B_1(0)} |b(x+ry,t) - \bar b(x,r,t)| \dd y
\end{align}and $\bar b (x,r,t)$ is chosen to equal zero if $\beta \in [-1,0)$, the average of $b$ over $B_{r}(x)$ if $\beta =0$, respectively $b(x,t)$ if $\beta \in (0,1]$, which is equivalent to \eqref{e:Morrey:def} (except for $\beta=1$). We give further details on the precise assumptions on $b$ in Section~\ref{s:active-scalar} below. Our main theorem in the case $\beta \in (-1,1]$ is given in Theorem~\ref{t:main} below, while the endpoint case $\beta = -1$ is addressed in Theorem~\ref{t:endpoint} (see also Remark~\ref{r:critical}).

\begin{thm} \label{t:main}
Assume $b: [0,T] \times \R^n \to \R^n$ is a divergence-free vector field such that $b \in L^p([0,T];M^{\beta})$ with $\beta \in (-1,1]$ and $p = 2/(1+\beta)$. Assume also that $u_0 \in C^\alpha$ for some some $\alpha \in (0,1)$. 
Then there exists a weak solution $u: [0,T] \times \R^n \to \R^n$ of the system
\begin{align}
u_t + b \cdot \grad u - \lap u &= \grad p \label{eq:1}\\
\nabla \cdot u &= 0 \label{eq:2}\\
u(x,0) &= u_0(x) \label{eq:3}
\end{align}
such that $u(x,t)$ is $C^\alpha$ in $x$ for all positive time $t\in(0,T]$. Moreover, we have the estimate
\[ [u(\cdot,t)]_{C^\alpha} \leq C [u_0]_{C^\alpha}, \]
for some positive universal constant $C=C(T,\alpha,\beta, [b]_{L_{t}^pM_{x}^{\beta}})$.
\end{thm}
\begin{remark}\label{r:scalar}
To the best of our knowledge, the result of Theorem~\ref{t:main} is new even if the pressure term was removed, and the scalar equation \eqref{e:no-pressure}
was considered instead. Indeed, the De Giorgi-Nash-Moser iteration scheme provides a $C^\alpha$ estimate for the solution for some small $\alpha$ only, whereas our result provides a $C^\alpha$ estimate for any $\alpha \in (0,1)$.
\end{remark}

\begin{remark}\label{r:critical}
The assumption $b \in L_{t}^pM_{x}^{\beta}$ implies a local smallness condition in the sense that $||b||_{L^p([t-\tau,t],M^{\beta})}$ becomes arbitrarily small as $\tau \to 0$, due to uniform integrability, but without any rate. This is not true for the endpoint case $p=\infty$. However, this is not the reason why we require $p<\infty$ in Theorem~\ref{t:main}, and in fact this local smallness plays no role in our proof. Indeed, the constant $C$ in the estimate of Theorem~\ref{t:main} depends only on the scaling-invariant norm of $b$, and not on any other feature of the vector field $b$, such as the modulus of continuity of the map $Q \mapsto \Vert b \Vert_{L^{p}_{t}M_{x}^{\beta}(Q)}$. Any argument that relies on the local smallness of $b$ would make the constants in the estimates depend on the rate at which the local norm of $b$ decays, and would hence be implicitly a subcritical result.
\end{remark}


In the particular case when $b=u$, Theorem \ref{t:main} becomes a no-blowup condition for solutions to Navier-Stokes equation. It says that if the norm of $u$ remains bounded in $L^p_{t}M_{x}^{\beta}$, then $u$ does not blow up on $\R^{3} \times [0,T]$. This is a scale invariant condition that is slightly more general (but has the same scaling) than the classical Ladyzhenskaya-Foias-Prodi-Serrin condition $u \in L^p_{t}L^q_{x}$ with $2/p+n/q \leq 1$ (cf.~\cite{FoiasProdi67,Ladyzhenskaya67,Prodi59,Serrin62}). Note that the endpoint case $(q,p)=(3,\infty)$, when $n=3$, was only treated recently in \cite{IskauriazaSereginSverak03} (see also \cite{KenigKoch10} and Theorem~\ref{t:endpoint} below for a related statement). The result of Theorem~\ref{t:main} for the full range of $p \in [1,\infty)$ may be new for the Navier-Stokes equations as well, though it is comparable to other available regularity criteria in terms of scaling critical norms of $u$ (cf.~\cite{ChanVasseur07,CheskidovShvydkoy10b,CheskidovShvydkoy10a,Kukavica08} and references therein).

One important difficulty for proving Theorem \ref{t:main} is to deal with the non-local pressure term on the right hand side of \eqref{e:linear-equation}. There are very few results of this kind available for equations with pressure terms. In \cite{Zhang06} the same equation \eqref{e:linear-equation} is considered and a Lipschitz estimate is shown under a sub-critical assumption on $b$ (which includes $b \in L^p_{t} L^q_{x}$ with $2/p+n/q\leq 1-\epsilon$, for any $\epsilon >0$).

The idea of the proof is to write the H\"older regularity condition of $u(\cdot,t)$ in integral form using a classical theorem of Campanato \cite{Campanato63}. Then we claim that these local integral estimates have a certain growth in time (in terms of integral estimates on $b$). In order to prove that these estimates hold for all time we argue by contradiction and look for the first point in which they would be invalidated. At that time we apply the equation and obtain a contradiction in a way that resembles maximum-principle-type arguments (see also \cite{Kiselev10,KiselevNazarovVolberg07} for the SQG equations). The integral representation of the H\"older modulus of continuity allows us to take advantage of the divergence-free condition and the integral bound on $b$. The divergence free condition on $u$ is used in the estimate for the gradient of pressure term. The general method of the proof introduced here seems to be new, as it may be applied to systems with pressure gradients, and we believe it may be applicable to other evolution equations in the future.

In section \ref{s:endpoint} we analyze the endpoint case $\beta=-1$. The method of this article is applicable in this case, but we need to impose an extra smallness condition on the vector field $b$ (cf.~\eqref{eq:borderline:1}--\eqref{eq:borderline:2} below).

We believe that the most important contribution of this article is the introduction of a new method to prove H\"older estimates for evolution equations. Nevertheless, we show an example of how Theorem \ref{t:main} can help prove that a nonlinear equation is well posed. Let us consider the following \emph{modified} energy critical Navier-Stokes equation in 3D (see~\cite{ConstantinIyerWu08} for a similar modified critically dissipative SQG equation)
\begin{align}
&\partial_{t} u + \left( (-\Delta)^{-1/4} u \cdot \nabla \right) u - \Delta u  = \nabla p \label{eq:m:nse:1}\\
&\dv u = 0 \label{eq:m:nse:2}
\end{align}that is, $b = (-\Delta)^{-1/4} u$ in \eqref{eq:1}--\eqref{eq:2}. It follows directly from Theorem~\ref{t:main} that the system \eqref{eq:m:nse:1}--\eqref{eq:m:nse:2} is well-posed in the classical sense. Indeed, the global existence of weak solutions $u \in L^{\infty}_{t} L^{2}_{x} \cap L^{2}_{t}\dot{H}_{x}^{1}$ is straightforward, as is the local existence of strong solutions. In particular, for any $T>0$, we have that $(-\Delta)^{-1/4} u$ is  a priori bounded in $L^{2}(0,T;H^{3/2}(\R^3))$, and by the Calder\'on-Zygmund theorem we have
\begin{align}
\left\Vert (-\Delta)^{-1/4} u \right\Vert_{L^{2}(0,T;BMO)} < \infty
\end{align} for any $T>0$. Therefore, by applying Theorem~\ref{t:main} with $\beta = 0$, we obtain that $u(\cdot,t) \in C^{\alpha}$ on $(0,T]$, given that $u(\cdot,0)=u_{0} \in C^{\alpha}$, where $\alpha \in (0,1)$ is arbitrary. From this estimate, it is easy to obtain higher regularity of $u$ by a standard bootstrap argument. 

Note that for the above system one could also use classical energy estimates at the level of vorticity, combined with Sobolev interpolation, to obtain the global well-posedness of the problem. On the other hand our method allows some extra flexibility in the relationship between $b$ and $u$. As explained above, when $b=(-\Delta)^{-1/4} u$ the system is well posed. Following essentially the same idea we can obtain using Theorem \ref{t:main} that the system is well posed for any of the following choices
\begin{itemize}
\item $b = a(x) (-\lap)^{-1/4} u + \grad q$ for any bounded function $a$ in $\R^3$ and $\grad q$ is the gradient of a scalar function that makes $b$ divergence free. In this case we apply the a priori estimate $u \in L^4 L^3$ that is obtained by interpolation from the energy inequality, and gives $b \in L^4 L^6$.
\item $b = \int k(x,y) u(y) \dd y$ where $k(x,-) \in L^{2n/(n+2)}$ for any $x \in \R^n$ and $k(-,y)$ is divergence free for any $y \in \R^n$.
\item $\lap b = \dv(u \otimes u) + \grad q$.
\end{itemize}

We plan to explore other applications of this method in the future.


\section{Preliminaries}
\label{s:preliminaries}
In this section we state a few introductory remarks about the weak and classical solutions to \eqref{eq:1}--\eqref{eq:2}, and recall a classical characterization of H\"older spaces in terms of local averages. Throughout the rest of the paper we will write $L^p L^q$ to denote $L_t^p L_x^q = L^p(0,T;L^q)$, and similarly $L^p M^\beta$ will be used instead of $L_t^p M_x^\beta$.

We  first prove Theorem \ref{t:main} assuming that the solution is classical (i.e. $C^2$ in space and $C^1$ in time). The important feature is that the a priori estimate \eqref{e:main} depends only on the assumptions of Theorem \ref{t:main} and not on any further smoothness assumptions on $b$ or $u$. Then we approximate any weak solution with classical solutions by using a mollification of $b$, and pass to the limit to obtain the result of Theorem \ref{t:main} in full generality.

\begin{defn}[\bf Weak Solutions]
If $b \in L^1_{loc}([0,T]\times \R^{n})$ is divergence-free, a function $u \in L^\infty ([0,T]\times \R^{n})$ is a weak solution of \eqref{e:linear-equation}, if it is weakly divergence-free, and for all smooth, divergence-free, compactly supported test functions $\varphi$ we have:
\[ \int_{\R^n} \varphi(x,T) u(x,T) \dd x + \int_{[0,T]\times \R^n} u \ \left( -\varphi_t + b \grad \varphi - \lap \varphi \right) \dd x \dd t = \int_{\R^n} \varphi(x,0) u(x,0) \dd x. \]
\end{defn}

The following proposition is standard.

\begin{prop} \label{p:weaklimit}
Let $b^\eps$ and $u^\eps$ be a sequence of smooth divergence-free vector fields. Assume that $u^\eps$ is a weak solution of \eqref{e:linear-equation} with drift velocity $b^\eps$. Assume also that $b^\eps \to b$ in $L^1_{loc} L^1_{uloc}$. Then, up to a subsequence, $u^\eps$ converges weakly to a weak solution of \eqref{e:linear-equation}.
\end{prop}

Using proposition \ref{p:weaklimit}, we immediately observe the following.

\begin{prop} \label{p:smoothing-theorem}
It is enough to prove Theorem \ref{t:main} assuming that $b$ is smooth and $u$ is a classical solution. 
\end{prop}

\begin{proof}[Proof of Proposition~\ref {p:smoothing-theorem}]
The assumption $b \in L^p M^{\beta}$ implies in particular that $b \in L^1_{loc} L^1_{uloc}$. Using a mollification argument, we consider a sequence of smooth vector fields $b^\eps$ converging strongly to $b$ in $L^1_{loc} L^1_{uloc}$. Moreover, we choose $b^\eps$ such that $||b^\eps||_{L^pM^{\beta}}$ is bounded uniformly with respect to $\eps$ (for example mollifying with a smooth function with fixed $L^1$ norm). For each of these vector fields, we solve the equation \eqref{e:linear-equation}, for instance using the mild formulation and Picard iteration, to obtain a smooth solution $u^\eps$. If the result of Theorem~\ref{t:main} is known for classical solutions, we would have that $u^\eps$ satisfies the estimate \eqref{e:main} uniformly in $\eps$. Note that in particular we also obtain $u^{\eps} \in L^{\infty}$. By Proposition \ref{p:weaklimit}, up to a subsequence, $u^\eps$ converges weakly to a weak solution $u$ of \eqref{e:linear-equation}, and therefore this solution $u$ satisfies \eqref{e:main} as well.
\end{proof}

In order to prove the main theorem, we use a local integral characterization of H\"older spaces. For this purpose, let $\varphi$ be a nonnegative, radially symmetric, smooth function supported in $B_1(0)$. Unless otherwise specified, the center of the unit ball $B_{1}$  in $\R^{n}$ shall be $0$. Let us also assume that $\int \varphi(y) \dd y = 1$. The following theorem (or a small variation of it) is proved in \cite{Campanato63}.

\begin{thm}[\bf Campanato's characterization of H\"older spaces] \label{t:campanato}
Let $f:\R^n \to \R^m$ be an $L^2$ function such that for all $r>0$ and $x \in \R^n$, there exists a constant $\bar{f}$ such that
\begin{align} \label{e:campanato}
\int_{B_1} |f(x+ry)-\bar{f}|^2 \varphi(y) \dd y \leq A^2 r^{2\alpha}
\end{align} for some positive constant $A$, and $\alpha \in (0,1)$. Then the function $f$ has a H\"older continuous representative such that
\[ |f(x)-f(y)| \leq B A |x-y|^\alpha\]
where the constant $B$ depends on dimension and $\alpha$ only. 
\end{thm}

The most natural choice of the constant $\bar f$ in the above theorem, for which the converse also holds. is to choose the average of $f$ in the ball
\[ \bar{f} = \int_{B_1} f(x+ry) \varphi(y) \dd y. \]
This is optimal in the sense that it minimizes the left hand side in \eqref{e:campanato} (see also \eqref{e:Morrey:def}).

The theorem of Campanato is interesting because it provides a non-obvious equivalence between a H\"older modulus of continuity, which is a priori a pointwise property, and averages of differences of the function, which is an integral property. This relation will allow us to exploit the divergence free nature of the vector fields $b$ and $u$ when estimating the evolution of a H\"older modulus of continuity.

\section{Evolution of a modulus of continuity}
\label{s:active-scalar}
We will prove that the solutions of \eqref{eq:1}--\eqref{eq:3} do not lose regularity by showing that they always satisfy a time dependent H\"older modulus of continuity. This modulus of continuity will evolve and deteriorate with time, but it will stay bounded. In order to take advantage appropriately of the divergence-free character of the vector field $u$, we use the integral characterization of the modulus of continuity. Let $\varphi$ be a radially symmetric weight supported in $B_1$ with mass one as in section \ref{s:preliminaries}. We denote the weighted mean of $u$ on $B_{r}(x)$ by
\begin{align} \label{e:baru}
\bar u(x,t,r) = \int_{B_1} u(x+ry,t) \varphi(y) \dd y.
\end{align}
The integral version of the modulus of the continuity of $u$ is then
\begin{align} \label{e:I}
\II(x,t,r) = \int_{B_1} |u(x+ry,t)-\bar u(x,t,r)|^2 \varphi(y) \dd y.
\end{align}
Due to Theorem \ref{t:campanato}, if we knew that 
\begin{align}
\II(x,t,r) \leq f(t)^2 r^{2\alpha}, \label{e:objective}
\end{align}
for some function $f(t) >0$, and all $r>0$, then $[u(\cdot,t)]_{C^\alpha} \leq C f(t)$ for some universal constant $C$. Our goal is to prove that estimate \eqref{e:objective} holds for all $t>0$, if it holds at $t=0$, for some function $f(t)$ to be chosen appropriately.

As discussed in the introduction, our assumptions on $b$ will be in terms of quantities  similar to $\II$, which are distinguished by the parameter $\beta \in [-1,1]$ as follows.
\begin{enumerate}
\renewcommand{\labelenumi}{(\roman{enumi})}
\item The Morrey-Campanaoto case. For $\beta \in [-1,0)$, let
\begin{align} \label{e:M}
\MM(x,t,r) = \int_{B_1} |b(x+ry,t)| \dd y = \frac{1}{r^n} \int_{B_r(x)} |b(z,t)| \dd z.
\end{align}
We assume that there exists a positive function $g \in L^{2/(1+\beta)}_{t}$ such that
\begin{align}\label{e:Morrey}
\sup_{x \in \RRd} \sup_{r>0} r^{-\beta} \MM(x,r,t) \leq g(t) \Leftrightarrow \Vert b(\cdot,t) \Vert_{M^{\beta}} \leq g(t)
\end{align}for all $t\geq 0$,  where $\Vert \cdot \Vert_{m^{s}}$ denotes the usual Morrey norm (cf.~\cite{Triebel92}).

\item The BMO case. For $\beta = 0$, we let
\begin{align} \label{e:B}
\MM(x,t,r)= \int_{B_1} |b(x+ry,t)-\bar b(x,r,t)| \dd y = \frac{1}{r^n} \int_{B_r(x)} |b(z,t) - \bar b(x,r,t)| \dd z,
\end{align}where
\begin{align}
  \bar b(x,r,t) = \frac{1}{r^n} \int_{B_r(x)} b(z,t) \dd z
\end{align}is the usual mean of $b$ on $B_{r}(x)$. We assume that there exists a positive function $g \in L^{2}_{t}$ such that
\begin{align}\label{e:BMO}
\sup_{x \in \RRd} \sup_{r>0} \MM(0,x,r,t) \leq g(t) \Leftrightarrow \Vert b(\cdot,t) \Vert_{BMO} \leq g(t)
\end{align}for all $t\geq 0$, where $\Vert \cdot \Vert_{BMO}$ denotes the norm on the space of functions with bounded mean oscillation.

\item The H\"older and Lipschitz cases. For $\beta \in (0,1]$, we consider
\begin{align} \label{e:H}
\MM(x,t,r) = \int_{B_1} |b(x+ry,t) - b(x,t)| \dd y = \frac{1}{r^n} \int_{B_r(x)} |b(z,t) - b(x,t)| \dd z.
\end{align}We assume that there exists a positive function $g \in L^{2/(1+\beta)}_t$ such that
\begin{align}\label{e:Holder}
\sup_{x\in \RRd} \sup_{r>0} r^{-\beta} \MM(x,t,r) \leq g(t) \Leftrightarrow [b(\cdot,t)]_{C^\beta} \leq g(t)
\end{align} for all $t\geq 0$, where $[\cdot]_{C^\beta}$ denotes the H\"older semi-norm. Note that $2/(1+\beta) \in [1,2)$ when $\beta \in (0,1]$.
\end{enumerate}

We shall prove that if \eqref{e:Morrey}, \eqref{e:BMO}, or respectively \eqref{e:Holder} holds, then we have $\II(x,t,r) < f(t)^2 r^{2\alpha}$ for all $t>0$. The proof is in the flavor of a maximum principle. We show that if the inequality is satisfied at $t=0$, it will be satisfied for all positive $t$. For the critically dissipative SQG equation, which as opposed to \eqref{eq:1}--\eqref{eq:2} is a  scalar equation, and  has an $L^{\infty}$ maximum principle, Kiselev, Nazarov, and Volberg~\cite{KiselevNazarovVolberg07}, use an argument in the same spirit, but where the breakdown is considered for the \emph{pointwise} modulus of continuity.

In order to prove Theorem~\ref{t:main}, assume there is a first time $t$ and some value of $x$ where the \emph{strict} modulus is invalidated, i.e. 
\begin{align}\label{e:breakdown}
\II(x,t,r) = f(t)^2 r^{2\alpha}.
\end{align} 
By Proposition \ref{p:smoothing-theorem}, we can assume that $u$ is a smooth function vanishing at infinity. Therefore the equality in \eqref{e:breakdown} of the modulus must be achieved at some $r>0$ and  $x\in \R^{n}$.

If we fix $t$ and $r$, the function $I$ achieves its maximum at $x$, and we obtain
\begin{align} \label{e:gradx}
0 = \grad_x \II = \int_{B_1} (u(x+ry,t)-\bar u(x,t,r)) \cdot (\grad_x u(x+ry,t)-\grad_x \bar u(x,t,r)) \varphi(y) \dd y.
\end{align}
Due to the definition of $\bar u$ \eqref{e:baru}, and the fact that $\grad_x \bar u(x,t,r)$ does not depend on $y$, we also have
\begin{align} \label{e:gradx-simplified}
0 = \int_{B_1} (u(x+ry,t)-\bar u(x,t,r)) \cdot \grad_x u(x+ry,t) \varphi(y) \dd y.
\end{align}
Since $\II < f(t)^2 r^{2\alpha}$ for all times prior to $t$, and since $(u-\bar u)\varphi$ has zero mean, we thus conclude
\begin{align}
2 f'(t) f(t) r^{2\alpha} \leq \partial_t \II &= \int_{B_1} (u(x+ry,t)-\bar u(x,t,r)) \cdot (\partial_t u(x+ry,t)-\partial_t \bar u(x,t,r)) \varphi(y) \dd y \\
&= \int_{B_1} (u(x+ry,t)-\bar u(x,t,r)) \cdot \partial_t u(x+ry,t) \varphi(y) \dd y.  \label{e:partialt}
\end{align}

The key to prove Theorem \ref{t:main} is to find an appropriate upper bound for the right hand side of \eqref{e:partialt} in terms of $f(t)$ and $\MM(r,t)$. Inserting the equation \eqref{e:linear-equation} in the right hand side of \eqref{e:partialt}, we obtain
\begin{align}
2 f'(t) f(t) r^{2\alpha} &\leq \int_{B_1} (u(x+ry,t)-\bar u(x,t,r)) \notag\\
& \qquad \cdot \Big(-b(x+ry,t) \cdot \nabla_x u(x+ry,t)+\lap_x u(x+ry,t)-\grad p(x+ry,t)\Big) \varphi(y) \dd y\notag \\
& = \AA + \DD + \PP. \label{e:whole-integral}
\end{align}
The following three lemmas give bounds to the three terms on the right side of \eqref{e:whole-integral}. The advection term $\AA$ is the simplest one to estimate. Observe that
\begin{align} \label{e:identity1}
\grad_x u(x+ry) = \frac 1r \grad_y u(x+ry).
\end{align}
This identity, together with the assumption $\dv b=0$, allows us to integrate by parts the gradient into the weight $\varphi$ and obtain a precise estimate for $\AA$.

The dissipative term turns out to be negative, but we must analyze it with care in order to obtain a precise lower bound on its absolute value. In fact, note that if $u$ is linear in $B_r(x)$ then $\DD=0$. We will obtain an estimate of $\DD$ that measures how much $u$ is forced to separate from a linear function, just from the values of $\II$ and $\partial_r \II$ at the point where the equality \eqref{e:breakdown} holds.

Lastly, we obtain an upper bound for the pressure term  $\PP$, comparable to the advection term $\AA$. This is to be expected since $\grad p$ is obtained from $b \cdot \grad u$ though an operator of order zero. However, the pressure estimate is more involved since the formula for the pressure is non-local and in order to obtain the desired estimate we need to take advantage of some cancellations that occur after integration by parts of Riesz kernels using that both $b$ and $u$ are divergence free.

We now carry out the estimates for the three terms on the right of \eqref{e:whole-integral} in the three lemmas below.

\begin{lemma}[\bf The advection term]
\label{l:advection}
Let $u$ and $b$ be as in the statement of Theorem~\ref{t:main}. Then we have
\begin{align}
\AA = \int_{B_1} (u(x+ry,t)-\bar u(x,t,r)) \cdot (-b(x+ry,t) \cdot \nabla_x u(x+ry)) \varphi(y) \dd y \leq Cr^{2\alpha - 1} f(t)^2 \MM(x,t,r)\label{e:term1}
\end{align}for all $\beta \in [-1,1]$, where $C$ is a positive constant depending only on $\alpha$, $\varphi$,  and $n$.
\end{lemma}
\begin{proof}[Proof of Lemma~\ref{l:advection}]
Using the identities \eqref{e:gradx-simplified}, \eqref{e:identity1}, and integrating by parts, we obtain from \eqref{e:term1} that
\begin{align} \label{e:term1-simplified}
\AA = \frac 1r \int_{B_1} |u(x+ry,t)-\bar u(x,t,r)|^2 (b(x+ry,t)-\bar b (x,r,t)) \cdot \grad \varphi(y) \dd y
\end{align}where $\bar b = \bar b (x,r,t)$ is a constant with respect to $y$, to be chosen suitably in the three cases for $\beta \in [-1,1]$, as discussed above. From identity \eqref{e:term1-simplified}, the H\"older inequality, and Theorem~\ref{t:campanato}, we directly obtain
\begin{align}
\AA \leq Cr^{2\alpha - 1} f(t)^2 \MM(x,t,r), \label{e:T1:Holder}
\end{align} for all $\beta \in[-1,1]$, where $C$ is a positive constant depending on $\alpha, n$, and $\varphi$ through $\sup_{B_{1}} |\nabla \varphi|$.
\end{proof} 

The second term corresponds to the viscosity and it is strictly negative, as we will show below. 

\begin{lemma}[\bf The dissipative term]\label{l:dissipative}
Let $u$ and $b$ be as in the statement of Theorem~\ref{t:main}. Then we have
\begin{align}
\DD &= \int_{B_1} \lap_x u(x+ry,t) \cdot (u(x+ry,t) - \bar u(x,t,r)) \varphi(y) \dd y \leq -c f(t)^2  r^{2\alpha-2}
\end{align}for all $r>0$, where $c$ is a sufficiently small positive constant, depending only on $n, \alpha$, and $\varphi$.
\end{lemma}

\begin{remark}
Note that the constant $c$ in Lemma \ref{l:dissipative} goes to zero as $\alpha \to 1$. This is the reason why Theorem \ref{t:main} works for $\alpha < 1$ only. 
\end{remark}

In order to prove Lemma~\ref{l:dissipative}, we need the following technical result, which relates the quantities $\DD$ and $\II$.
\begin{lemma}\label{l:D:lower}
For a fixed $x$, we have
\begin{equation} \label{e:funct-ineq}
\DD(r) \geq c \frac{ (2\II(r) - r \II'(r))^2 }{r^2 \II(r)}
\end{equation}for some sufficiently small positive constant $c$.
\end{lemma}

For simplicity we omited the $x$ dependence in \eqref{e:funct-ineq}. Note that if $f$ is a linear function, then $\DD=0$ and $\II=Cr^2$ for some constant $C$. Moreover, in this case $r \II' = 2 \II$. We see that if $(r \II' - 2 \II)$ is non zero, then the function $u$ cannot be linear. The inequality \eqref{e:funct-ineq} gives a precise quantitative version of this fact.

We note that the identity
\begin{align}
\int_{B_1} |f(y)-c|^2 \varphi(y) \dd y = \int_{B_1} \int_{B_1} (f(y)-f(z))^2 \varphi(y) \varphi(z) \dd y \dd z \label{e:identity:squares}
\end{align}holds for any constant $c$, in particular for $c = \bar f$, giving equivalent definitions for $\II$ and $\DD$ in terms of double-integrals. The proof of identity \eqref{e:identity:squares} is straightforward, and hence omitted.

With formula \eqref{e:identity:squares} in mind, we prove the following lemma, which is exactly the case $r=1$ for \eqref{e:funct-ineq}. The proof of Lemma~\ref{l:D:lower} follows for all other values of $r$ from Lemma~\ref{l:functional-inequality} by scaling.

\begin{lemma} \label{l:functional-inequality}
Let $f:B_1 \to \R^n$ be any $H^1$ function. There is a constant $C$ depending only on $n$ and $\varphi$ such that
\begin{align*}
&\iint_{B_1 \times B_1} (f(y)-f(z))^2 \varphi(y) \varphi(z) \dd y \dd z - \iint_{B_1 \times B_1} (y \cdot \grad f(y) - z \cdot \grad f(z)) \cdot (f(y)-f(z)) \varphi(y) \varphi(z) \dd y \dd z \\
& \qquad \leq C \left( \iint_{B_1 \times B_1} (\grad f(y)-\grad f(z))^2 \varphi(y) \varphi(z) \dd y \dd z \right)^{1/2} \left(\iint_{B_1 \times B_1} (f(y)-f(z))^2 \varphi(y) \varphi(z) \dd y \dd z \right)^{1/2}.
\end{align*}
\end{lemma}

\begin{proof}[Proof of Lemma~\ref{l:functional-inequality}]
We start by writing $f(y)-f(z)$ as an integral of $\grad f$ along the segment between $z$ and $y$. Thus, we have
\begin{align*}
& \iint_{B_1 \times B_1} (f(y)-f(z))^2 \varphi(y) \varphi(z) \dd y \dd z \\
& \qquad \qquad =  \iint_{B_1 \times B_1} \int_0^1 (\grad f(sy+(1-s)z)\cdot (y-z)) \cdot (f(y)-f(z)) \varphi(y) \varphi(z) \dd s \dd y \dd z.
\end{align*}Substituting in the first term of the left hand side, we obtain
\begin{align*}
& \iint_{B_1 \times B_1} (f(y)-f(z))^2 \varphi(y) \varphi(z) \dd y \dd z\\
&\qquad \qquad \qquad - 2\iint_{B_1 \times B_1} (y \cdot \grad f(y) - z \cdot \grad f(z)) \cdot (f(y)-f(z)) \varphi(y) \varphi(z) \dd y \dd z \\
&\qquad  = \iint_{B_1 \times B_1} \int_0^1 \Big((\grad f(sy+(1-s)z)-\grad f(y))\cdot y - (\grad f(sy+(1-s)z)-\grad f(y))\cdot z\Big) \\
& \qquad \qquad \qquad  \qquad \qquad \times(f(y)-f(z)) \varphi(y) \varphi(z) \dd s \dd y \dd z\\
& \qquad \leq \II^{1/2} \EE^{1/2},
\end{align*}
where $\II$ is as defined in \eqref{e:I}, and $\EE$ is defined as
\begin{align*}
\EE &= \iint_{B_1 \times B_1} \int_0^1 \Big((\grad f(sy+(1-s)z)-\grad f(y))\cdot y - (\grad f(sy+(1-s)z)-\grad f(y))\cdot z\Big)^2 \dd s \dd y \dd z.
\end{align*}
The quantity $\EE$ is a bounded quadratic functional with respect to the vector field $\grad f$ that vanishes whenever $\grad f$ is constant. Therefore, there is a constant $C$ such that
\begin{align*}
\EE \leq C \iint_{B_1 \times B_1} (\grad f(y) - \grad f(z))^2 \varphi(y) \varphi(z) \dd y \dd z = C\, \DD,
\end{align*}
which concludes the proof.
\end{proof}

\begin{proof}[Proof of Lemma~\ref{l:dissipative}]
We recall that at the first point of equality in \eqref{e:breakdown}, the function $\II$ achieves its maximum as a function of $x$. Then we have $\lap_x \II \leq 0$, and hence
\begin{align*}
0 \geq& \int_{B_1} (\lap_x u(x+ry,t) - \lap_x \bar u(x,t,r)) \cdot (u(x+ry,t) - \bar u(x,t,r)) \varphi(y) \dd y \notag\\
&\qquad + \int_{B_1} |\grad u(x+ry,t)-\grad_x \bar u(x,t,r)|^2 \varphi(y) \dd y.
\end{align*}
Therefore, since $(u-\bar u)\varphi$ has zero mean and $\lap_{x} \bar u$ does not depend on $y$, we obtain the inequality
\begin{align} \label{e:laplacian-inequality}
\DD  \leq - \int_{B_1} |\grad u(x+ry,t)-\grad_x \bar u(x,t,r)|^2 \varphi(y) \dd y. 
\end{align}
According to \eqref{e:laplacian-inequality}, $\DD$  is negative, but we must estimate how negative it is. It is convenient to rewrite the formula for $\II(x,t,r)$ using \eqref{e:identity:squares} as\begin{align}
\II(x,t,r) &= \int_{B_1} |u(x+ry,t)-\bar u(x,t,r)|^2 \varphi(y) \dd y \notag \\
&= \iint_{B_1 \times B_1} |u(x+ry,t)-u(x+rz,t)|^2 \varphi(y) \varphi(z) \dd y \dd z.  \label{e:II}
\end{align}
At the point where $f(t)^{2} r^{2\alpha}=\II(x,t,r)$ for the first time, we have $\grad_x \II=0$ and  $\lap_x \II\leq0$, so that
\begin{align*}
0 &\geq \lap_x \II(x,t,r) \\
&= 2 \iint_{B_1 \times B_1} |\grad_x u(x+ry,t)-\grad_x u(x+rz,t)|^2 \\
& \qquad \qquad + (\lap_x u(x+ry,t)-\lap_x u(x+rz,t)) \cdot (u(x+ry,t)-u(x+rz,t)) \varphi(y) \varphi(z) \dd y \dd z \\
&= 2 \DD + 2 \iint_{B_1 \times B_1} |\grad_x u(x+ry,t)-\grad_x u(x+rz,t)|^2 \varphi(y) \varphi(z) \dd y \dd z,
\end{align*}
where $\DD$ is given by \eqref{e:laplacian-inequality}. We thus have proven the inequality
\begin{align} 
\DD &\leq -\iint_{B_1 \times B_1} |\grad_x u(x+ry,t)-\grad_x u(x+rz,t)|^2 \varphi(y) \varphi(z) \dd y \dd z \notag \\
& = -\frac 1{r^2} \iint_{B_1 \times B_1} |\grad_y u(x+ry,t)-\grad_z u(x+rz,t)|^2 \varphi(y) \varphi(z) \dd y \dd z. \label{e:inequality-for-D}
\end{align}
The right hand side is clearly negative unless $u$ is an affine function. We now need to estimate how negative it is, in terms of $\II$ and $\partial_r \II$. Note that at the first point of equality $\II=f(t)^2 r^{2\alpha}$, we must also have $\partial_r \II = 2\alpha f(t)^2 r^{2\alpha-1}$. We compute $\partial_r \II$ as
\begin{align} \label{e:partialr}
\partial_{r}\II & = 2 \int_{B_1} (u(x+ry,t)-\bar u(x,t,r))\ y \cdot \grad_x u(x+ry,t) \varphi(y) \dd y \notag\\
& = 2 \iint_{B_1 \times B_1} (u(x+ry,t)-u(x+rz,t)) \cdot ( y \cdot \grad_x u(x+ry,t) - z \cdot \grad_x u(x+rz,t)) \varphi(y) \varphi(z) \dd y \dd z \notag\\
& = \frac 2 r \iint_{B_1 \times B_1} (u(x+ry,t)-u(x+rz,t)) \cdot ( y \cdot \grad_y u(x+ry,t) - z \cdot \grad_z u(x+rz,t)) \varphi(y) \varphi(z) \dd y \dd z.
\end{align}
From the expression \eqref{e:II} and \eqref{e:partialr}, we can apply Lemma~\ref{l:functional-inequality} to obtain
\begin{align*} 
(2\II - r \partial_r \II) \leq C \II^{1/2} \left(\iint_{B_1 \times B_1} |\grad_y u(x+ry,t)-\grad_z u(x+rz,t)|^2 \varphi(y) \varphi(z) \dd y \dd z \right)^{1/2}.
\end{align*}
Recalling the inequality \eqref{e:inequality-for-D}, we obtain
\begin{align} \label{e:bound-below-for-D}
\DD \leq -\frac{c}{r^2} \frac{(2\II - r \partial_r \II)^2}{\II} = -c f(t)^2 (1-\alpha)^2 r^{2\alpha-2}
\end{align}for some positive constant $c$, depending only on $\alpha, n$, and $\varphi$.
\end{proof}

Lastly, we bound the pressure term $\PP$ arising on the right side of \eqref{e:whole-integral}.

\begin{lemma}[\bf The pressure term]\label{l:pressure}
Let $u$ and $b$ be as in the statement of Theorem~\ref{t:main}. Then we have
\begin{align}
\PP = \int_{B_{1}} ( u(x+ry,t) - \bar u (x,r,t)) \cdot \nabla p(x+ry,t) \varphi(y) \dd y \leq C f(t)^2 g(t) r^{2\alpha + \beta - 1}
\end{align}for all $r>0$ and $\beta \in [-1,1]$, where $C$ is a positive constant, depending on $\alpha,\beta,n$, and $\varphi$.
\end{lemma}
\begin{proof}[Proof of Lemma~\ref{l:pressure}]
Since all estimates in this section hold for a fixed time $t>0$, we omit the time dependence of all functions. Recall that the function $\bar b = \bar b(x,r)$ (which is constant respect to $y$) is chosen to be $\bar b = b(x)$ in the H\"older case, the average of $b$ over $B_r(x)$ in the BMO case, or $\bar b = 0$ in the Morrey-Campanato case. 

In order to estimate the third term in \eqref{e:whole-integral}, let us analyze the identity \eqref{e:grad-p}. Since $u$ is divergence free, we have
\begin{align}
 \grad_x p(x+ry) &= \grad_x (-\lap)_x^{-1} \dv_x (b(x+ry) \cdot \grad_x u(x+ry)) \notag \\
 &= \frac 1r \grad_y (-\lap_y)^{-1} \dv_y \Big((b(x+ry)- \bar b(x,r)) \cdot \grad_y (u(x+ry)-\bar u(x,r))\Big).\label{e:gradp2}
\end{align}
We have the pressure term $\PP$ equal to
\begin{align*}
& \int  (u(x+ry) - \bar u(x,r)) \cdot \grad_x p(x+ry) \varphi(y) \dd y \notag \\
&= \frac 1r \int (u(x+ry) - \bar u(x,r)) \cdot \grad_y (-\lap_y)^{-1} \dv_y \Big(  (b(x+ry)-\bar b (x,r)) \cdot \grad_y (u(x+ry)-\bar u(x,r)) \Big) \varphi(y) \dd y.
\end{align*}
We integrate the $\grad_y$ by parts and use that $u$ is divergence free to obtain that $\PP$ equals
\begin{align*}
\frac 1r \int (-\lap_y)^{-1} \dv_y \Big(  (b(x+ry)-\bar b(x,r)) \cdot \grad_y (u(x+ry)-\bar u(x,r)) \Big)  (u(x+ry) - \bar u(x,r)) \cdot \grad \varphi(y) \dd y.
\end{align*}
Using that $b$ and $u$ are divergence free, we re-write the above identity as
\begin{align}
\PP&= \frac 1r \int (-\lap_y)^{-1} \partial_i \partial_j \Big( (b(x+ry)-\bar b(x,r))_i (u(x+ry)-\bar u(x,r))_j \Big)  (u(x+ry) - \bar u(x)) \cdot \grad \varphi(y) \dd y \notag \\
&= \frac 1r \int R_i R_j \Big( (b(x+ry)-\bar b(x,r))_i (u(x+ry)-\bar u(x,r))_j \Big) (u(x+ry) - \bar u(x,r)) \cdot \grad \varphi_y(y) \dd y \notag \\
&= \frac 1r \int (b(x+ry)-\bar b(x,r))_i (u(x+ry) -\bar u(x,r))_j  R_i R_j \Big( (u(x+ry) - \bar u(x,r)) \cdot \grad \varphi(y) \Big) \dd y, \label{e:Riesz:1}
\end{align}where the Riesz transforms are taken with respect to the $y$ variable.

The third factor inside the integral is given by Riesz transforms in $y$ of the function
\begin{align}
\psi_{x,r}(y) = \psi(y) = (u(x+ry)-\bar u(x,r)) \cdot \grad \varphi(y) = \dv_y\Big((u(x+ry)-\bar u(x,r)) \varphi(y) \Big),
\end{align}since $u$ is divergence free. The function $\psi$ is supported in $B_1$  and we have $\int_{B_1} \psi(y) \dd y = 0$. Moreover, since by assumption $I(x,r) \leq C f(t) r^\alpha$ for some $\alpha \in (0,1)$, we have by Theorem~\ref{t:campanato} that $\Vert u(x+ry)-\bar u(x,r)\Vert_{L^\infty_y(B_1)} \leq C f(t) r^\alpha$, and therefore $\Vert \psi \Vert_{L^\infty_y(B_1)} \leq C f(t) r^\alpha$, uniformly in $x$. In fact, $\psi$ is also $C^\alpha$, with $C^\alpha$ norm bounded by $C f(t) r^\alpha$ (note the scaling in $y$).

Gathering these bounds together, we see that $R_i R_j(\psi)$ must be bounded in $L^\infty(\RRd)$ by $C f(t) r^\alpha$, uniformly in $x$. Indeed, the Riesz transforms are bounded on $C^\alpha$ and on $L^2$ (since $\psi$ is supported on $B_1$ and is bounded there, its $L^2$ norm is also finite), and therefore $R_i R_j(\psi) \in L^2(\RRd) \cap C^\alpha(\RRd) \supset L^\infty(\RRd)$.

The Riesz transforms of functions with compact support are not compactly supported. The decay of the Riesz transform of a compactly supported function is normally of order $-n$. However, in this case since the function $\psi$ has integral zero, and since it is a derivative, we have that $R_i R_j (\psi)$ decays like $|y|^{-n-2}$ for $|y|$ large. To see this, let $K_{ij}$ be the Kernel associated with the Riesz transform and, using that $\dv u=0$, we compute
\begin{align}
  R_{ij}(\psi)(y) = \int \psi(z) K_{ij}(y-z) \dd z &= \int \varphi(z)  ( u(x+rz) - \bar u(x,r) )  \cdot \grad K_{ij}(y-z) \dd z\notag\\
  & = \int \varphi(z)  ( u(x+rz) - \bar u(x,r) )  \cdot \Big( \grad K_{ij}(y-z) - \grad K_{ij}(y) \Big) \dd z
\end{align}in principal value sense, since $\int \varphi ( u - \bar u)=0$. Letting $y$ be such that $|y| > 2$, we obtain that $|\nabla^2 K_{ij}(\xi)| \leq C |y|^{n+2}$, for all $\xi$ that lies between $y$ and $y-z$, where $C$ is a sufficiently large dimensional constant, gives the desired decay in $|y|$.

To summarize, we have proved that
\begin{align}
R_i R_j [(u(x+ry) - \bar u(x,r)) \cdot \grad \varphi(y)] \leq \begin{cases}
C f(t) r^\alpha & \text{if } |y|<2 \\
\frac{C f(t) r^\alpha }{|y|^{n+2}} & \text{if } |y| \geq 2
\end{cases}\label{e:Riesz:2}
\end{align}where $C$ is a universal constant that does not depend on $x$. Recall that by our assumption and Campanato's theorem we also have $(u(x+ry)-\bar u(x,r)) \leq C f(t) (ry)^\alpha$. Therefore, we obtain from \eqref{e:Riesz:1} and \eqref{e:Riesz:2} the estimate
\begin{align}
\PP \leq Cf(t)^2  r^{2\alpha-1} \left( \int_{B_2} |b(x+ry) - \bar b(x,r)| \dd y + \int_{\R^n \setminus B_2} |b(x+ry)- \bar b(x,r)| \frac{1}{|y|^{n+2-\alpha}} \dd y \right)\label{e:T3:bound:1}.
\end{align}
It is here where it is necessary to make the distinction between the H\"older and Morrey-Camapanato cases for the a priori assumptions on $b$, by making the specific choices for $\bar b$.

\subsubsection*{The H\"older case}
If $\bar b (x,r) = b(x)$, it is clear that
\begin{align}
  \int_{B_2} |b(x+ry) - \bar b(x)| \dd y \leq C \MM(x,r/2) \leq C g(t) r^\beta \label{e:Holder:origin}
\end{align}for  $\beta\in(0,1]$. To bound the tail of the integral arising from the Riesz transforms, we first change variables $r y = z$ so that
\begin{align}
  \int_{\R^n \setminus B_2} |b(x+ry)- \bar b(x)| \frac{1}{|y|^{n+2-\alpha}} \dd y &= r^{2-\alpha} \int_{\RRd \setminus B_{2r}} |b(x+z) - \bar b(x)| \frac{1}{|z|^{n+2-\alpha}} \dd z\notag\\
  & = r^{2-\alpha} \int_{2r}^{\infty} \int_{\partial B_\rho} |b(x+z ) - \bar b(x)| \dd \sigma(z) \frac{1}{\rho^{n+2-\alpha}} \dd \rho\notag\\
  & = r^{2-\alpha} \int_{2r}^{\infty} \frac{\partial}{\partial \rho} \left(\int_{B_\rho} |b(x+z ) - \bar b(x)| \dd \sigma(z)\right) \frac{1}{\rho^{n+2-\alpha}} \dd \rho\notag\\
  & \leq C r^{2-\alpha}  \int_{2r}^{\infty} \left( \rho^n \MM(x,\rho) \right) \frac{1}{\rho^{n+3-\alpha}} \dd \rho\notag\\
  & \leq C r^{2-\alpha} g(t) \int_{2r}^{\infty} \frac{1}{\rho^{3-\alpha-\beta}} \leq C g(t) r^\beta. \label{e:Holder:tail}
\end{align}From \eqref{e:T3:bound:1}, \eqref{e:Holder:origin}, and \eqref{e:Holder:tail}, we obtain
\begin{align}
  \PP \leq C f(t)^2 g(t) r^{2\alpha + \beta - 1}.\label{e:T3:Holder}
\end{align}

\subsubsection*{The BMO case}
In this case we have $\bar b(x,r) = (1/r^n) \int_{B_r(x)} b(z) \dd z$. The difference with the H\"older case lies in the tail of the integral due to the Riesz transform. We use the following classical fact about BMO functions: the difference between the mean on $B_r$ an $B_{\lambda r}$ is bounded by $2^n \ln (1+\lambda)$ times the BMO norm, for all $\lambda >1$. We split the integral in dyadic cylinders to find a bound for each and sum.
\begin{align*}
\int_{\R^n \setminus B_2} &|b(x+ry)-\bar b(x,r)| \frac 1 {|y|^{n+2-\alpha}} \dd y = \sum_{k=1}^\infty \int_{B_{2^{k+1}} \setminus B_{2^k} } |b(x+ry)-\bar b(x,r)| \frac 1 {|y|^{n+2-\alpha}} \dd y \\
&\leq \sum_{k=1}^\infty \int_{B_{2^{k+1}} \setminus B_{2^k} } \left(|b(x+ry)-\bar b(x,2^{k+1} r)| + |\bar b(x,2^{k+1} r)-\bar b(x,r)| \right) \frac 1 {2^{k(n+2-\alpha)}} \dd y \\
&\leq \sum_{k=1}^\infty \frac 1 {2^{k(2-\alpha)}} ||b||_{BMO} + \frac 1 {2^{k(n+2-\alpha)}} \int_{B_{2^{k+1}r} \setminus B_{2^k r} } |\bar b(x,2^{k+1} r)-\bar b(x,r)| \dd y \\
&\leq \sum_{k=1}^\infty \frac 1 {2^{k(2-\alpha)}} ||b||_{BMO} + \frac 1 {2^{k(n+2-\alpha)}} C 2^{kn} \log(1+2^{k+1}) ||b||_{BMO} \leq C ||b||_{BMO}.
\end{align*}
Thus, the tail of the integral in the bound of $\PP$ \eqref{e:T3:bound:1} is bounded by $C ||b||_{BMO} \leq C g(t)$, as well as the fist term in \eqref{e:T3:bound:1}. Therefore, in this case ($\beta=0$) we also obtain
\[ \PP \leq C f(t)^2 g(t) r^{2\alpha + \beta -1}.\]

\subsubsection*{The Morrey-Campanato case}
In this case we have $\bar b(x,r) = 0$. The same proof as in the H\"older case above, but with $\bar b=0$, shows that
\begin{align}
  \PP \leq C f(t)^2 g(t) r^{2\alpha + \beta - 1}\label{e:T3:Morey:negative}
\end{align}when $\sup_{r>0} r^{-\beta} \MM(x,r) <\infty$,  and $\beta\in[-1,0)$.
\end{proof}

Once we have estimated the three terms on the right side of \eqref{e:whole-integral}, the proof of the main theorem is concluded as follows.

\begin{proof}[Proof of Theorem~\ref{t:main}]
Since the vector drift $b$ is a priori assumed to lie in $M^{\beta}$, for some $\beta \in(-1,1]$, with $[b(\cdot, t)]_{M^{\beta}} \leq g(t)$, and $g \in L^{2/(1+\beta)}_{t}$, we obtain from \eqref{e:T1:Holder}, \eqref{e:bound-below-for-D}, and \eqref{e:T3:Holder} the estimate
\begin{align}
  f'(t) f(t) r^{2\alpha} \leq C_* f(t)^2 g(t) r^{2\alpha + \beta - 1} - c_*  f(t)^2 r^{2\alpha-2}
\end{align}which holds for some sufficiently large constant $C_*$, and some sufficiently small positive constant $c_*$. The above estimate implies
\begin{align}
  \frac{f'(t)}{f(t)} \leq C_* g(t) r^{\beta-1} - c_*r^{-2} \leq \bar{C} g(t)^{2/(1+\beta)} \label{e:log:F}
\end{align}for some positive constant $\bar C = \bar C(C_*,c_*,\beta)$. The last inequality was obtained maximizing the expression with respect to $r$.

On the other hand, if $g \in L_t^{2/(1+\beta)}$, we can choose $f$ to be the solution to the ODE
\[f'(t) = 2 \bar C g(t)^{2/(1+\beta)} f(t), \]
which contradicts the above inequality and makes it impossible for the H\"older modulus of continuity to ever be invalidated. Note that the above ODE has the explicit solution
\[f(t) = e^{\left(\int_0^t C g(s)^{2/(1+\beta)} \dd s \right)} f(0), \]
which stays bounded for all $t$.
\end{proof}

\section{The endpoint case $\beta = -1$}\label{s:endpoint}
For the three-dimensional Navier-Stokes equations, i.e. $b=u$, obtaining regularity of the solutions in the endpoint case $ u\in L^{\infty} L^{3}$  is highly non-trivial, and this issue was only settled recently by Iskauriaza, Seregin, and Sverak in \cite{IskauriazaSereginSverak03} (see also \cite{KenigKoch10} for the case $L^{\infty} H^{1/2}$). The case $(p,q) = (\infty,n)$ on the Ladyzhenskaya-Foias-Prodi-Serrin scale is of particular importance as it is the scaling-critical space for the initial data, and it gives the borderline space (on the Lebesgue scale) for constructing solutions via the Picard iteration scheme (so-called mild solutions). Note that $L^{3} \subset M^{-1}$.

It turns out that for the linear system \eqref{eq:1}--\eqref{eq:3}, the proof given above fails in the case $b\in L^{\infty}M^{-1}$, as the constant $\bar C$ blows up as $\beta\rightarrow -1$. The corresponding result for the Navier-Stokes equations (as in \cite{IskauriazaSereginSverak03} or \cite{KenigKoch10}) relies essentially on the nonlinear structure of the equation. In order to obtain a result in this direction for the liner equation \eqref{e:linear-equation}, with the method of this article, we need to impose an extra smallness condition.

\begin{thm}\label{t:endpoint}
Let $T>0$ be arbitrary. Assume that $b$ is a divergence-free vector field in $L^{\infty}([0,T];M^{-1})$, and let $u_{0}\in L^{2} \cap C^{\alpha}$ for some $\alpha \in (0,1)$. There exists a positive constant $\epsilon>0$ such that if for all $t\in(0,T]$ there exists $r_{\ast}(t) > 0 $ with 
\begin{align}
\sup_{t\in[0,T]}\sup_{x\in \R^{n}} \sup_{0<r<r_{\ast}(t)} r \int_{B_{1}} |b(x+ry,t)| \dd y \leq \epsilon \label{eq:borderline:1}
\end{align}
and 
\begin{align}
\int_{0}^{T} \frac{1}{r_{\ast}(t)^{2}} \dd t < \infty, \label{eq:borderline:2}
\end{align}
then $u(\cdot,t) \in C^{\alpha}(\R^{n})$ for all $t\in (0,T]$.
\end{thm}
\begin{remark}
If $\Vert b \Vert_{L^{\infty} M^{-1}} \leq \epsilon$, the theorem holds trivially. Additionally, note that for any $\phi \in L^{3}(\R^{3})$, or any $\phi$ in the closure of $C_{0}^{\infty}$ in $M^{-1}$,  we have 
\begin{align*}
\lim_{r\rightarrow 0} r \int_{B_{1}} |\phi(x+ry)| \dd y = 0,
\end{align*} for all $x$, and hence \eqref{eq:borderline:1} holds for some $r_\ast > 0$. Therefore, if $b$ is continuous in time with values in $L^3$, or piecewise continuous with arbitrarily large jumps, or if all jumps are of size smaller than $\eps/2$, the conditions \eqref{eq:borderline:1}--\eqref{eq:borderline:2} are automatically satisfied, with $r_{\ast}(t)$ being a sufficiently small constant, proving regularity of the solution (see also \cite[Theorem 3.1]{CheskidovShvydkoy10a} for a similar result in the critical Besov space). Theorem~\ref{t:endpoint} states that if the drift velocity $b(\cdot,t)$ is  nicely behaved at sufficiently small scales $r_{\ast}(t)$, and if these scales do not go to $0$ \textit{too fast} at any point $t\in [0,T]$, i.e. $r_{\ast}(s)$ cannot vanish at $t$ with a rate faster than $\sqrt{t-s}$, then the solution is regular until $T$. Note that with respect to time regularity, the conditions of Theorem~\ref{t:endpoint} above are stronger than merely $b\in L^{\infty}_tL^{3}_x$, but weaker than $C_tL^{3}_x$.
\end{remark}
\begin{proof}[Proof of Theorem~\ref{t:endpoint}]
From \eqref{e:T1:Holder}, \eqref{e:bound-below-for-D}, \eqref{e:Holder:origin}, and \eqref{e:Holder:tail}, we obtain
\begin{align}
2 f'(t) f(t) r^{2\alpha} &\leq C_{\ast} r^{2\alpha -1} f(t)^{2}\MM(x,t,r) - c_{\ast} f(t)^{2} r^{2\alpha-2} \notag\\
&\qquad +  C_{\ast} f(t)^{2} r^{2\alpha - 1} \left( \MM(x,t,r/2) + r^{2-\alpha}  \int_{2r}^{\infty} \MM(x,t,\rho)\frac{1}{\rho^{3-\alpha}} \dd \rho\right), \label{e:end:1}
\end{align}
where $\MM(x,t,r)  = \int_{B_{1}} |b(x+ry,t)| \dd y$, and by assumption of the theorem we have
\begin{align}
\sup_{0<r<r_{\ast}(t)} r \MM(x,t,r) \leq \epsilon \label{e:end:2}
\end{align}
and
\begin{align}
\sup_{r>0} r \MM(x,t,r) \leq \Vert b \Vert_{L^{\infty}_{t}M_{1}^{-1}} = B < \infty \label{e:end:3}
\end{align}
uniformly in $x,t$. Inserting the bounds \eqref{e:end:2}--\eqref{e:end:3} into \eqref{e:end:1}, we have
\begin{align}
2 r^{2}\frac{d}{dt} \log f(t)  &\leq C_{\ast} r \MM(x,t,r) - c_{\ast}+  C_{\ast} r \MM(x,t,r/2) \notag\\
&\qquad +  C_{\ast} r^{3-\alpha} \left( \int_{2r}^{r_{\ast}} (\rho \MM(x,t,\rho)) \frac{1}{\rho^{4-\alpha}} \dd \rho +   \int_{r_{\ast}}^{\infty} (\rho \MM(x,t,\rho)) \frac{1}{\rho^{4-\alpha}} \dd \rho\right) \label{e:end:4}
\end{align}for all $r>0$. To bound the right side of \eqref{e:end:4}, we distinguish the cases $r/r_{\ast} \leq \delta$, and $r/r_{\ast} > \delta$, where we let
\begin{align}
\delta = \min \left\{ \left( \frac{\epsilon}{B} \right)^{1/(3-\alpha)} , \frac 12 \right\}.
\end{align}
Indeed, if $r/r_{\ast} \leq \delta$, we have $2r < r_{\ast}$ and $B (r/r_{\ast})^{3-\alpha} \leq B \delta^{3-\alpha} \leq \epsilon$, so that by \eqref{eq:borderline:1} implies
\begin{align}
2 r^{2}\frac{d}{dt} \log f(t)  &\leq C_{\ast} \epsilon - c_{\ast} + 2 C_{\ast} \epsilon + C_{\ast} r^{3-\alpha} C_{\alpha} \left(\epsilon r^{\alpha - 3} + B r_{\ast}^{\alpha-3} \right) \notag\\
&\leq C_{\ast} (3 + 2 C_{\alpha} ) \epsilon - c_{\ast} , \label{e:end:5}
\end{align}for some positive constant $C_{\alpha} >0$, and for all $0 < r \leq \delta\, r_{\ast}$. Therefore, if we choose $\epsilon$ as 
\begin{align}
\epsilon = \frac{c_{\ast}}{C_{\ast} ( 3 + 2 C_{\alpha})}, \label{eq:epsilon}
\end{align}then we obtain from \eqref{e:end:5} that
\begin{align}
\frac{d}{dt} \log f(t) \leq 0 \label{e:end:6}
\end{align}for all $t \in [0,T]$, and all $0 < r \leq\delta\, r_{\ast}(t)$. On the other hand, if $r > \delta\, r_{\ast}(t)$, we bound the right side of  \eqref{e:end:4} by making use of \eqref{e:end:3}, namely
\begin{align}
\frac{d}{dt} \log f(t) &\leq \frac{1}{2 r^{2}} \left( 3 C_{\ast} B + C_{\ast} B r^{3-\alpha}\int_{2r}^{\infty} \frac{1}{\rho^{4-\alpha}} \right)\leq  \frac{C_{\ast} B (3 + C_{\alpha})}{2 \delta^{2} r_{\ast}(t)^{2}}\label{e:end:7}
\end{align}for some $C_{\alpha} >0$, for all $r > \delta\, r_{\ast}(t)$. The proof of the theorem is then concluded since we may choose $f$ as
\begin{align}
 f(t) = f(0) \exp \left( \frac{C_{\ast} B (3 + C_{\alpha})}{\delta^{2}} \int_{0}^{t}\frac{1}{r_{\ast}(s)^{2}} \right)
\end{align}which is finite for all $t\leq T$ by \eqref{eq:borderline:2}, and it contradicts \eqref{e:end:6}--\eqref{e:end:7}.
\end{proof}

\section*{Acknowledgment}
Luis Silvestre was partially supported by NSF grant DMS-1001629 and the Sloan Foundation.

\bibliographystyle{amsplain}   
\bibliography{as}             
\index{Bibliography@\emph{Bibliography}}%

\end{document}